\documentclass[10pt,reqno]{amsart}

\usepackage{amsthm,amsmath,amstext,amssymb,amscd,euscript, mathrsfs, dsfont,multicol,times,enumerate,subfig,sidecap}
\usepackage{esint}
\usepackage{xcolor}
\usepackage{mathtools}
\usepackage{bm}
\usepackage{bbm}
\usepackage{todonotes}

\usepackage[colorlinks=true, pdfstartview=FitV, linkcolor=blue, citecolor=red, urlcolor=blue]{hyperref} 

\newtheorem{thm}{Theorem}[section]

\newtheorem{lem}[thm]{Lemma}
\newtheorem{prop}[thm]{Proposition}
\theoremstyle{definition}
\newtheorem{defn}[thm]{Definition}

\newtheorem{assumption}[thm]{Assumption}
\newtheorem{rem}[thm]{Remark}

\newcommand\cM{\mathcal{M}}

\DeclareMathOperator*{\esssup}{ess\,sup}

\newcommand{\supp}{\mathrm{supp}}

\newcommand{\aint}{-\hspace{-0.40cm}\int}

\def\XXint#1#2#3{{\setbox0=\hbox{$#1{#2#3}{\int}$ }
\vcenter{\hbox{$#2#3$ }}\kern-.58\wd0}}

\makeatletter
\def\@tocline#1#2#3#4#5#6#7{\relax
  \ifnum #1>\c@tocdepth 
  \else
    \par \addpenalty\@secpenalty\addvspace{#2}%
    \begingroup \hyphenpenalty\@M
    \@ifempty{#4}{%
      \@tempdima\csname r@tocindent\number#1\endcsname\relax
    }{%
      \@tempdima#4\relax
    }%
    \parindent\z@ \leftskip#3\relax \advance\leftskip\@tempdima\relax
    \rightskip\@pnumwidth plus4em \parfillskip-\@pnumwidth
    #5\leavevmode\hskip-\@tempdima
      \ifcase #1
       \or\or \hskip 1em \or \hskip 2em \else \hskip 3em \fi%
      #6\nobreak\relax
    \dotfill\hbox to\@pnumwidth{\@tocpagenum{#7}}\par
    \nobreak
    \endgroup
  \fi}
\makeatother

\begin{document}

\title[A regularity theory for second-order parabolic partial differential equations in weighted mixed-norm Sobolev-Zygmund spaces]{A regularity theory for second-order parabolic partial differential equations in weighted mixed-norm Sobolev-Zygmund spaces}

\author{Jae-Hwan Choi}
\address{School of Mathematics, Korea Institute for Advanced Study, 85 Hoegi-ro, Dongdaemun-gu, Seoul, 02455, Republic of Korea}
\email{jhchoi@kias.re.kr}
\thanks{J.-H. Choi was supported by a KIAS Individual Grant (MG102701) at Korea Institute for Advanced Study.}

\author{Junhee Ryu}
\address{School of Mathematics, Korea Institute for Advanced Study, 85 Hoegi-ro, Dongdaemun-gu, Seoul, 02455, Republic of Korea} \email{junhryu@kias.re.kr}
\thanks{J. Ryu was supported by a KIAS Individual Grant (MG101501) at Korea Institute for Advanced Study and by the National Research Foundation of Korea (NRF) grant funded by the Korea government (Ministry of Science and ICT) (No. RS-2026-25477288).}



\begin{abstract}
We develop an optimal regularity theory for parabolic partial differential equations in weighted mixed-norm Sobolev-Zygmund spaces.
The results extend the classical Schauder estimates to coefficients that are merely measurable in time and to the critical case of integer-order regularity.
In addition, non-zero initial data are treated in the optimal trace space via a sharp trace theorem.
\end{abstract}

\maketitle

\section{Introduction and Main results} 
\subsection{Goal and setting}
In this paper, we study the existence, uniqueness, and regularity of solutions to the second-order parabolic partial differential equation (PDE)
\begin{equation}
\label{25.12.06.15.30}
    \partial_t u(t,x)
    = \sum_{i,j=1}^{d} a_{ij}(t,x) D_{x^i x^j} u(t,x)
      + \sum_{i=1}^{d} b_i(t,x) D_{x^i} u(t,x)
      + c(t,x) u(t,x) + f(t,x),
\end{equation}
posed on $(0,T)\times\mathbb{R}^d$, with non-zero initial data.  
Our framework is based on the weighted mixed-norm Sobolev-Zygmund space $L_p((0,T),w\,\mathrm{d}t;\Lambda^{\gamma}(\mathbb{R}^d))$, where $\Lambda^\gamma(\mathbb{R}^d)$ is the H\"older-Zygmund space of order $\gamma\in(0,\infty)$ and the weight $w$ belongs to the Muckenhoupt class $A_p(\mathbb{R})$.
For $p\in(1,\infty]$, the mixed-norm is given by
$$
    \|u\|_{L_p((0,T),w\,\mathrm{d}t;\Lambda^{\gamma}(\mathbb{R}^d))}
    :=
    \begin{cases}
        \left(\displaystyle\int_0^T \|u(t,\cdot)\|_{\Lambda^{\gamma}(\mathbb{R}^d)}^p
        w(t)\,\mathrm{d}t\right)^{1/p}, & p\in(1,\infty),\\[2mm]
        \displaystyle \esssup_{t\in(0,T)} \|u(t,\cdot)\|_{\Lambda^{\gamma}(\mathbb{R}^d)}, & p=\infty.
    \end{cases}
$$

We begin by recalling the definition of H\"older-Zygmund spaces and Muckenhoupt's weight class.
\begin{defn}[H\"older-Zygmund space]
\label{25.10.12.23.39}
Let $\gamma\in(0,\infty)$ and $p\in[1,\infty]$.
    The space $\Lambda_p^{\gamma}(\mathbb{R}^d)$ is defined by
    $$
    \Lambda_p^{\gamma}(\mathbb{R}^d):=\{f\in L_{\infty}(\mathbb{R}^d):\|f\|_{\Lambda_p^{\gamma}(\mathbb{R}^d)}<\infty\}.
    $$
    Here the norm $\|f\|_{\Lambda_p^{\gamma}(\mathbb{R}^d)}$ is given by
    \begin{eqnarray*}
    \|f\|_{\Lambda_p^{\gamma}(\mathbb{R}^d)}:=\|f\|_{L_{\infty}(\mathbb{R}^d)}+\|f\|_{\dot{\Lambda}_p^{\gamma}(\mathbb{R}^d)},
    \end{eqnarray*}
    where
    $$
    \|f\|_{\dot{\Lambda}_p^{\gamma}(\mathbb{R}^d)}:=\begin{cases}
        \left(\int_{\mathbb{R}^d}\|\mathcal{D}_{h}^{[\gamma]^{-}}f\|_{L_{\infty}(\mathbb{R}^d)}^p\frac{\mathrm{d}h}{|h|^{d+p\gamma}}\right)^{1/p},\quad &p\in[1,\infty),\\
        \sup\limits_{h\in\mathbb{R}^d}\frac{\|\mathcal{D}_{h}^{[\gamma]^{-}}f\|_{L_{\infty}(\mathbb{R}^d)}}{|h|^{\gamma}},\quad &p=\infty.
        \end{cases}
    $$
Here $\mathcal{D}_hf(x):=f(x+h)-f(x)$, $\mathcal{D}_h^nf(x)=\mathcal{D}_h^{n-1}(\mathcal{D}_hf)(x)$, and $[\gamma]^{-}$ is the smallest integer strictly larger than $\gamma$.
We let $\Lambda^{\gamma}(\mathbb{R}^d):=\Lambda_{\infty}^{\gamma}(\mathbb{R}^d)$ and $\dot{\Lambda}^{\gamma}(\mathbb{R}^d):=\dot{\Lambda}_{\infty}^{\gamma}(\mathbb{R}^d)$.
\end{defn}

\begin{defn}[Muckenhoupt's class]
For $p \in (1,\infty)$, let $A_p = A_p(\mathbb{R})$ be the class of all 
nonnegative and locally integrable functions $w:\mathbb{R}\to [0,\infty)$ satisfying
\begin{equation*}
    [w]_{A_p} := \sup_{r>0, t\in\mathbb{R}} \left(\aint_{t-r}^{t+r} w(s)\mathrm{d}s \right)\times\left(\aint_{t-r}^{t+r} w(s)^{-\frac{1}{p-1}} \mathrm{d}s \right)^{p-1} < \infty.
\end{equation*}
\end{defn}

\subsection{Motivation and background}
The regularity theory for linear parabolic equations of the form
\eqref{25.12.06.15.30} has been extensively studied.  
A classical cornerstone is the \emph{Schauder theory}
which roughly states that
$$
    a_{ij},b_i,c,f \in C^{\alpha/2,\alpha}_{t,x}, \quad \alpha\in(0,1)
    \;\Longrightarrow\;
    u \in C^{1+\alpha/2,\,2+\alpha}_{t,x}.
$$
However, the classical Schauder theory faces two fundamental limitations:
\begin{enumerate}[1.]
    \item it does not apply when the coefficients $a_{ij},b_i$, and $c$ are merely
    \emph{measurable in time};
    \item it breaks down at the \emph{critical index} $\alpha=1$.
\end{enumerate}

To overcome the first limitation, Brandt \cite{brandt1969interior}
introduced the \emph{partial Schauder theory}, in which regularity is measured in mixed-norm spaces $L_\infty(\mathbb{R},\mathrm{d}t; C^\alpha(\mathbb{R}^d))$, thereby accommodating time-measurable coefficients while preserving spatial H\"older continuity. Lorenzi \cite{Lor00} and Krylov \cite{krylov2002calderon,krylov2003parabolic} subsequently generalized Brandt’s approach to the mixed-norm spaces
$L_p(\mathbb{R},\mathrm{d}t; C^\alpha(\mathbb{R}^d))$ for $p\in(1,\infty]$.
Following their work, several further extensions have been developed.
Boccia and Krylov \cite{MR3287783} extended the theory to higher-order parabolic systems.
Stinga and Torrea \cite{stinga2021holder} proved weighted Sobolev-Schauder estimates for second-order parabolic equations in the mixed-norm setting 
$$
L_p(\mathbb{R}, w(t)\mathrm{d}t; C^\alpha(\mathbb{R}^d)),\quad w\in A_p(\mathbb{R}).
$$
Despite these advances, classical H\"older spaces remain inadequate in the critical case of \emph{integer} regularity $\alpha=1$, where the Schauder estimate is known to fail; see, for instance, \cite[Chapter 2.2]{MR4560756}.

A natural way to address the second limitation is to work within the H\"older-Zygmund spaces $\Lambda^\gamma$.
These spaces coincide with H\"older spaces for non-integer orders, $\Lambda^\gamma = C^\gamma$ for $\gamma\notin\mathbb{N}$, but become strictly larger for integer orders (see Proposition \ref{25.12.06.17.46}):
$$
    C^\gamma \subsetneq \Lambda^\gamma, \qquad \gamma\in\mathbb{N}.
$$
Within this framework, Kim \cite{kim2018lp} established regularity estimates in $L_p((0,T),\mathrm{d}t;\Lambda^\alpha(\mathbb{R}^d))$, but only for equations with zero initial data.
This assumption is likewise adopted in the results of Krylov \cite{krylov2002calderon,krylov2003parabolic},
Boccia-Krylov \cite{MR3287783}, and Stinga-Torrea \cite{stinga2021holder}.
Recently, the first author of the present paper \cite{C24} resolved this limitation by developing a regularity theory in the weighted mixed-norm spaces
\begin{equation}
\label{25.12.11.14.01}
    L_p((0,T), w(t)\mathrm{d}t; \Lambda^\gamma(\mathbb{R}^d)),
    \qquad w\in A_p(\mathbb{R}),\; p\in(1,\infty],
\end{equation}
allowing for non-zero initial data through the use of the optimal trace theorem.

While the aforementioned results focus on time-dependent coefficients, the corresponding theory for space-time-dependent coefficients within the Zygmund framework has not been fully established.
The extension to variable coefficients is motivated not only by theoretical completeness but also by its relevance to nonlinear parabolic PDEs.
A pertinent example is the two-dimensional incompressible Navier-Stokes equation in the vorticity form
\begin{equation}
\label{25.12.15.15.47}
    \partial_t\omega + u\cdot\nabla_x\omega = \Delta_x\omega,
    \qquad
    u = \nabla_x^{\perp}\Delta_x^{-1}\omega.
\end{equation}
It is well-known that for initial vorticity $\omega_0\in L_1(\mathbb{R}^2)\cap L_{\infty}(\mathbb{R}^2)$, there exists a unique global weak solution satisfying $\omega\in L_{\infty}((0,T);L_1(\mathbb{R}^2)\cap L_{\infty}(\mathbb{R}^2))$,
as established by Ben-Artzi and Brezis \cite{MR1308857,MR1308858}.
Moreover, Ben-Artzi \cite{MR1308857} also proved
$$
\partial_tD^{k}_x\omega(t,\cdot),\,D^k_x\omega(t,\cdot)\in L_{\infty}(\mathbb{R}^2), \quad \forall k\in\mathbb{N},t>0.
$$
We note that this regularization is recovered via the partial Schauder theory.
Treating the velocity field $u$ as a transport coefficient that is bounded in time yet regular in space motivates the study of linear theories in such mixed-norm classes. 
Such a perspective provides a useful analytical tool for linearized models arising from nonlinear PDEs.

Recently, Wei et al. \cite{wei2024new} studied parabolic PDEs with coefficients depending on both time and space, and obtained regularity estimates for solutions in Lebesgue-H\"older-Dini spaces.
Their results, however, are derived under zero initial data and unweighted time norms, and do not extend to Zygmund regularity, weighted mixed-norm settings, or the delicate \emph{integer}-order case.
Consequently, despite the well-developed theory for time-dependent coefficients, the Zygmund theory for fully variable coefficients $a_{ij}(t,x)$ remains largely open.
Motivated by this gap, we study equation \eqref{25.12.06.15.30} in the weighted mixed-norm space \eqref{25.12.11.14.01}, thereby extending the existing partial Schauder and Zygmund regularity theories to parabolic PDEs with variable coefficients, including the integer-order regime $\gamma\in\mathbb{N}$.

We also refer to several works that investigate related regularity questions in different settings, such as those based on Zygmund spaces or the partial Schauder theory; see \textit{e.g.}, \cite{MR2764915,MR3938318,MR3910092,MR4289894,MR4991198}.
Lastly, we expect that the present approach may be useful in developing partial Schauder regularity theory for other linear and nonlinear equations under suitable structural assumptions, including the models studied in
\cite{GLX25,HJ26,HLZ26,MZ24,WYY2026}.

\subsection{Main results}

We begin by introducing the weighted mixed-norm Sobolev-Zygmund spaces that serve as the natural setting for our solutions.  
These spaces combine spatial Zygmund regularity with time-dependent $L_p$ integrability under the Muckenhoupt weights, which is essential for treating parabolic PDEs
whose coefficients are merely measurable in time.
\begin{defn}[Solution space]
\label{26.04.14.13.37}
Let $\gamma\in(0,\infty)$, $p\in(1,\infty]$, and $w\in A_p(\mathbb{R})$.
We put $w\equiv1$ when $p=\infty$.
\begin{enumerate}[(i)]
    \item For $0<S<T<\infty$, we define
\begin{equation*}
    \mathbf{\Lambda}_{p,w}^{\gamma}(S,T) := \begin{cases}
    L_p((S,T), w\, \mathrm{d}t; \Lambda^{\gamma}(\mathbb{R}^d)),\quad &\text{if } p\in(1,\infty),\\
    L_{\infty}((S,T), \mathrm{d}t; \Lambda^{\gamma}(\mathbb{R}^d)),\quad &\text{if }p=\infty,
    \end{cases}
\end{equation*}
equipped with the norm defined by
\begin{equation*}
    \|u\|_{\mathbf{\Lambda}_{p,w}^{\gamma}(S,T)}:= 
    \left(\int_S^T \|u(t,\cdot)\|_{\Lambda^{\gamma}(\mathbb{R}^d)}^p \,w(t) \mathrm{d}t \right)^{1/p}1_{p\in(1,\infty)}+\esssup_{t\in(S,T)}\|u(t,\cdot)\|_{\Lambda^{\gamma}(\mathbb{R}^d)}1_{p=\infty}. 
\end{equation*}
When $S=0$, we simply define $\mathbf{\Lambda}_{p,w}^{\gamma}(T):=\mathbf{\Lambda}_{p,w}^{\gamma}(0,T)$.

    \item We define $\Lambda_{p}^{\gamma,w}(\mathbb{R}^d)$ to be the set of all continuous functions $u$ defined on $\mathbb{R}^d$ satisfying
$$
\|u\|_{\Lambda_{p}^{\gamma,w}(\mathbb{R}^d)}:=\|u\|_{L_{\infty}(\mathbb{R}^d)}+\|u\|_{\dot{\Lambda}_{p}^{\gamma,w}(\mathbb{R}^d)}<\infty,
$$
where
$$
\|u\|_{\dot{\Lambda}_{p}^{\gamma,w}(\mathbb{R}^d)}:=\left(\int_{\mathbb{R}^d}\frac{\|\mathcal{D}_h^{[\gamma]^{-}}u\|_{L_{\infty}(\mathbb{R}^d)}^p}{W(|h|^{2})^{-1}}\frac{\mathrm{d}h}{|h|^{d+p\gamma}}\right)^{1/p}1_{p\in(1,\infty)}+\|u\|_{\dot{\Lambda}^{\gamma}(\mathbb{R}^d)}1_{p=\infty},
$$
and $W(\lambda):=\int_0^{\lambda}w(s)\mathrm{d}s$.

    \item The space $\mathbf{H}_{p,w}^{\gamma+2}(S,T)$ is the set of all $u\in \mathbf{\Lambda}_{p,w}^{\gamma+2}(S,T)$ for which there exist $u_0\in \Lambda_{p}^{\gamma+2,w}(\mathbb{R}^d)$ and $f\in \mathbf{\Lambda}_{p,w}^{\gamma}(S,T)$ such that
    $$
    u(t,x)=u_0(x)+\int_S^tf(s,x)\mathrm{d}s
    $$
    for all $(t,x)\in [S,T)\times\mathbb{R}^d$.
    We also define
    $$
    \|u\|_{\mathbf{H}_{p,w}^{\gamma+2}(S,T)}:=\|u\|_{\mathbf{\Lambda}_{p,w}^{\gamma+2}(S,T)}+\|u_0\|_{\Lambda_{p}^{\gamma+2,w}(\mathbb{R}^d)}+\|f\|_{\mathbf{\Lambda}_{p,w}^{\gamma}(S,T)}.
    $$
\end{enumerate}
\end{defn}

\begin{rem}
The space $\Lambda^{\gamma+2,w}_p(\mathbb{R}^d)$ in Definition \ref{26.04.14.13.37}-(ii) is introduced precisely so as to coincide with the natural \emph{initial trace space} associated with the weighted parabolic class.
In other words, the operator $\mathrm{T}$ defined by 
$$
\mathrm{T}:u(t,x)\mapsto u(0,x)
$$ 
is the trace operator from $\mathbf{H}_{p,w}^{\gamma+2}(0,T)$ to $\Lambda^{\gamma+2,w}_p(\mathbb{R}^d)$. 
In Theorem \ref{25.11.01.13.31}, we will show that it is a bounded linear operator and admits a bounded right inverse.
Since the trace operator $\mathrm{T}$ evaluates the function at $t=0$, the relevant time-weight information is accumulated from the initial time. 
This is why the function parameter is given by the primitive of $w$ anchored at $0$, rather than by a translated quantity such as $\int_t^{t+\lambda} w(s)\,\mathrm{d}s$. 
Moreover, the factor $W(|h|^2)$ is consistent with the parabolic scaling $t\sim |h|^2$.

As a model example, if $w(t)=t^\alpha$ near $t=0$ with $\alpha>-1$, then
$$
W(\lambda)=\frac{1}{\alpha+1}\lambda^{\alpha+1},
\qquad
W(|h|^2)\sim |h|^{2(\alpha+1)},
\qquad
\Lambda_{p}^{\gamma,w}(\mathbb{R}^d)=B_{\infty,p}^{\gamma-\frac{2(1+\alpha)}{p}}(\mathbb{R}^d),
$$
(see \textit{e.g.} \cite[Theorem 1.1]{MV14} or \cite[Remark 1.8]{CLSW25}).
Hence, the trace norm reflects precisely how the weighted time-regularity accumulates near the initial boundary.
\end{rem}

\begin{rem}
It can be verified directly that the pair $(u_0, f) \in \Lambda_{p}^{\gamma+2,w}(\mathbb{R}^d) \times \mathbf{\Lambda}_{p,w}^{\gamma}(S,T)$ associated with $u$ is unique.  
Hence, we may unambiguously write $f = \partial_t u$.
\end{rem}

We now state the assumptions on the coefficients and present the main result.
The conditions below ensure uniform ellipticity of the operator and spatial $\Lambda^\gamma$–regularity of the coefficients.
\begin{assumption}[$\gamma$]
\label{25.10.12.23.59}
There exist positive constants $\nu$ and $K$ satisfying the following conditions:
\begin{enumerate}
    \item For all $(t,x)\in(0,\infty)\times\mathbb{R}^d$,
\begin{equation}
\label{25.10.20.19.33}
    \nu|\xi|^2\leq \sum_{i,j=1}^{d}a_{ij}(t,x)\xi_i\xi_j,\quad \forall \xi\in\mathbb{R}^d.
\end{equation}
    \item The coefficients $a_{ij}$, $b_i$, and $c$ are measurable, and satisfy 
\begin{equation*}
\sum_{i,j=1}^{d}\|a_{ij}\|_{L_{\infty}((0,\infty);\Lambda^{\gamma}(\mathbb{R}^d))}+\sum_{i=1}^{d}\|b_{i}\|_{L_{\infty}((0,\infty);\Lambda^{\gamma}(\mathbb{R}^d))}+\|c\|_{L_{\infty}((0,\infty);\Lambda^{\gamma}(\mathbb{R}^d))}\leq K.
\end{equation*}
\end{enumerate}
\end{assumption}

\begin{thm}
\label{25.10.20.19.28}
Let $T \in (0,\infty)$, $p \in (1,\infty]$, $\gamma > 0$, and let $w \in A_p(\mathbb{R})$ with $[w]_{A_p}\leq K_0$ (put $w\equiv1$ and $K_0=1$ when $p=\infty$).
Suppose that Assumption \ref{25.10.12.23.59} $(\gamma)$ holds.
Then, for any pair $(u_0, f) \in \Lambda_{p}^{\gamma+2,w}(\mathbb{R}^d) \times \mathbf{\Lambda}_{p,w}^{\gamma}(T)$, there exists a unique solution 
$u \in \mathbf{H}_{p,w}^{\gamma+2}(T)$ to \eqref{25.12.06.15.30} in $(0,T) \times \mathbb{R}^d$, with the initial condition $u(0,x) = u_0(x)$.
Moreover, the following estimate holds:
\begin{equation}
\label{25.10.30.12.00}
    \|u\|_{\mathbf{H}_{p,w}^{\gamma+2}(T)}
    \leq N \big(
        \|u_0\|_{\Lambda_{p}^{\gamma+2,w}(\mathbb{R}^d)}
        + \|f\|_{\mathbf{\Lambda}_{p,w}^{\gamma}(T)}
    \big),
\end{equation}
where $N = N(d,\gamma,K_0,K,\nu,T,p)$.
\end{thm}





\section{Preliminaries on H\"older-Zygmund Spaces}
\subsection{Characterizations of H\"older-Zygmund Spaces}
In this subsection, we recall several equivalent descriptions of the H\"older-Zygmund spaces $\Lambda_p^{\gamma}(\mathbb{R}^d)$.  
\begin{prop}
\label{25.10.20.14.01}
Let $\gamma\in(0,\infty)$ and $p\in(1,\infty]$.
\begin{enumerate}[(i)]
    \item (Equivalence with Besov spaces) Let $\{\Delta_j\}_{j\in\mathbb{Z}}$ be standard Littlewood-Paley projection operators defined by
$$
\Delta_j f := \psi_j \ast f,\qquad 
\mathcal{F}[\psi_j](\xi):=\mathcal{F}[\psi](2^{-j}\xi),\qquad 
\supp(\mathcal{F}[\psi])=\{\xi\in\mathbb{R}^d:1/2\leq|\xi|\leq 2\},
$$
with $\mathcal{F}[\psi]\geq 0$ and $\sum_{j\in\mathbb{Z}}\mathcal{F}[\psi](2^{-j}\xi)=1$ for $\xi\neq 0$.
Then 
$$
\|f\|_{\dot{\Lambda}_p^{\gamma}(\mathbb{R}^d)}\simeq \|f\|_{\dot{B}_{\infty,p}^{\gamma}(\mathbb{R}^d)}
$$
where
$$
\|f\|_{\dot{B}_{\infty,p}^{\gamma}(\mathbb{R}^d)}:= \begin{cases}
    \left(\sum_{j\in\mathbb{Z}}2^{j\gamma p}\|\Delta_jf\|_{L_{\infty}(\mathbb{R}^d)}^p\right)^{1/p},\quad &\text{if } p\in(1,\infty),\\
    \sup_{j\in\mathbb{Z}}2^{j\gamma}\|\Delta_jf\|_{L_{\infty}(\mathbb{R}^d)},\quad &\text{if } p=\infty.
\end{cases}
$$
    \item (Characterization via derivatives) 
If $\gamma=n+\delta$, $n\in\mathbb{N}\cup\{0\}$, and $\delta\in(0,1]$, then $$
    \|f\|_{\Lambda_p^{\gamma}(\mathbb{R}^d)}\simeq \sum_{|\alpha|\leq n}\|D^{\alpha}f\|_{L_{\infty}(\mathbb{R}^d)}+\sum_{|\beta|=n}\|D^\beta f\|_{\dot{\Lambda}_p^{\delta}(\mathbb{R}^d)}.
    $$
    The equivalence constants depend only on $d,\gamma$, and $p$.
\end{enumerate}
\end{prop}
\begin{proof}
The result $(i)$ is well-known, \textit{e.g.}, see \cite[Theorem 1.5, Theorem 1.6]{CLSW23}.
    
For $(ii)$,  iterating elementary difference estimates,    we have
    $$
    \|\mathcal{D}^{[\gamma]^{-}}_hf\|_{L_{\infty}(\mathbb{R}^d)}\leq |h|^{n}\sum_{|\alpha|=n}\|\mathcal{D}_h^{[\delta]^{-}}D^{\alpha}f\|_{L_{\infty}(\mathbb{R}^d)}.
    $$
    Therefore,
    $$
    \|f\|_{\dot{\Lambda}_p^{\gamma}(\mathbb{R}^d)}\leq N(n,p) \sum_{|\alpha|=n}\|D^{\alpha}f\|_{\dot{\Lambda}_p^{\delta}(\mathbb{R}^d)},
    $$
    which proves the “$\leq$” direction.

For the converse inequality, note that for any multi-index $\alpha$,
    \begin{equation}
    \label{25.09.25.11.53}
    \|\Delta_jD^{\alpha}f\|_{L_{\infty}(\mathbb{R}^d)}\leq N2^{j|\alpha|}\|\Delta_{j}f\|_{L_{\infty}(\mathbb{R}^d)}.
    \end{equation}
    Hence,
    \begin{equation}
    \label{25.09.25.11.37}
    \begin{aligned}
    \|D^{\alpha}f\|_{L_{\infty}(\mathbb{R}^d)}\leq& N\sum_{j\in\mathbb{Z}}2^{j|\alpha|}\|\Delta_jf\|_{L_{\infty}(\mathbb{R}^d)}\\
    \leq& N\left(\|f\|_{L_{\infty}(\mathbb{R}^d)}+\sum_{j=1}^{\infty}2^{j|\alpha|}\|\Delta_jf\|_{L_{\infty}(\mathbb{R}^d)}\right).
    \end{aligned}
    \end{equation}
    For $|\alpha|\leq n$ and $p\in(1,\infty)$, H\"older's inequality gives
    \begin{equation}
    \label{25.09.25.11.36}
    \begin{aligned}
    \sum_{j=1}^{\infty}2^{j|\alpha|}\|\Delta_jf\|_{L_{\infty}(\mathbb{R}^d)}&\leq \left(\sum_{j=1}^{\infty}2^{j\gamma p}\|\Delta_jf\|_{L_{\infty}(\mathbb{R}^d)}^p\right)^{1/p}\left(\sum_{j=1}^{\infty}2^{j(|\alpha|-\gamma)p'}\right)^{1/p'}\\
    &\leq N\left(\sum_{j=1}^{\infty}2^{j\gamma p}\|\Delta_jf\|_{L_{\infty}(\mathbb{R}^d)}^p\right)^{1/p}\leq N\|f\|_{\dot{B}_{\infty,p}^{\gamma}(\mathbb{R}^d)},
    \end{aligned}
    \end{equation}
    where $p'=p/(p-1)$.
    Combining \eqref{25.09.25.11.37}, \eqref{25.09.25.11.36}, and $(i)$, we obtain the desired bound for $p\in(1,\infty)$.
    When $p=\infty$, we instead use the trivial bound
$$
\sum_{j=1}^\infty 2^{j|\alpha|}\|\Delta_j f\|_{L_\infty}
\leq \sup_{j\in\mathbb{Z}}\big(2^{j\gamma}\|\Delta_j f\|_{L_\infty}\big)\,
\sum_{j=1}^{\infty}2^{j(|\alpha|-\gamma)}
\leq N\|f\|_{\dot{B}_{\infty,\infty}^{\gamma}(\mathbb{R}^d)}.
$$
Thus, for all $p\in(1,\infty]$,
$$
    \|D^{\alpha}f\|_{L_{\infty}(\mathbb{R}^d)}\leq N\|f\|_{\Lambda^{\gamma}_p(\mathbb{R}^d)},\quad \forall |\alpha|\leq n.
    $$    
    Finally, by $(i)$ and \eqref{25.09.25.11.53}, for $|\beta|=n$ we have
    \begin{align*}
        \|D^{\beta}f\|_{\dot{\Lambda}^{\delta}_p(\mathbb{R}^d)}&\simeq \left(\sum_{j\in\mathbb{Z}}2^{j\delta p}\|\Delta_jD^{\beta}f\|_{L_{\infty}(\mathbb{R}^d)}^p\right)^{1/p}\leq N\left(\sum_{j\in\mathbb{Z}}2^{j\gamma p}\|\Delta_jf\|_{L_{\infty}(\mathbb{R}^d)}^p\right)^{1/p}
    \end{align*}
    with the obvious modification when $p=\infty$ (replace $\ell_p$ by $\ell_\infty$). 
    Summing over all $|\beta|=n$ completes the proof.
\end{proof}

If $p=\infty$ and $\gamma>0$ is a non-integer, then the Zygmund space $\Lambda^{\gamma}(\mathbb{R}^d)$ coincides with the classical H\"older space $C^{\gamma}(\mathbb{R}^d)$.  
At integer orders $\gamma\in\mathbb{N}$, however, this identification fails.  
The next result states the precise inclusions between the corresponding spaces.
\begin{prop}
\label{25.12.06.17.46}
For $n\in\mathbb{N}$,
    $$
    C^n(\mathbb{R}^d)\subsetneq C^{n-1,1}(\mathbb{R}^d)\subsetneq \Lambda^{n}(\mathbb{R}^d)\subsetneq C_{\log}^{n-1,1}(\mathbb{R}^d),
    $$
    where
    $$
    \|f\|_{C^{n-1,1}_{\log}(\mathbb{R}^d)}:=\sum_{|\alpha|\leq n-1}\|D^\alpha f\|_{L_{\infty}(\mathbb{R}^d)}+\sum_{|\beta|=n-1}\sup_{|x-y|<1}\frac{|D^{\beta}f(x)-D^{\beta}f(y)|}{|x-y|\log_2(2/|x-y|)}.
    $$
\end{prop}
\begin{proof}
The strict inclusion 
$C^{n}(\mathbb{R}^d) \subsetneq C^{n-1,1}(\mathbb{R}^d)$ 
is well known, so we focus on the remaining inclusions.  
By Proposition \ref{25.10.20.14.01}-(ii),
$$
\|f\|_{\Lambda^{n}(\mathbb{R}^d)}
\simeq
\sum_{|\beta| = n-1}
\|D^{\beta} f\|_{\Lambda^{1}(\mathbb{R}^d)}
+
1_{n \geq 2}
\sum_{|\alpha| \leq n-2}
\|D^{\alpha} f\|_{L_{\infty}(\mathbb{R}^d)}.
$$
Thus, it suffices to establish the result for $n = 1$, namely
\begin{equation}
\label{25.12.03.13.36}
C^{0,1}(\mathbb{R}^d)
\subsetneq
\Lambda^{1}(\mathbb{R}^d)
\subsetneq
C^{0,1}_{\log}(\mathbb{R}^d).
\end{equation}

\textbf{1.} $C^{0,1}(\mathbb{R}^d)\subsetneq \Lambda^1(\mathbb{R}^d)$.

Let $\zeta\in C_c^{\infty}(\mathbb{R}^d)$ be a nonnegative cut-off function near zero, and $f(x):=x^1\log_2(2/|x|)$. 
Then for $|h|<1$,
        $$
        |f(h)-f(0)|=|h^1|\log_2(2/|h|),
        $$
        which implies that $f\zeta\not\in C^{0,1}(\mathbb{R}^d)$.
        However, $f\zeta\in \Lambda^1(\mathbb{R}^d)$, which easily follows from
        \begin{align*}
            |f(h)+f(-h)-2f(0)|=0\leq |h|.
        \end{align*}

        \textbf{2.} $\Lambda^1(\mathbb{R}^d)\subsetneq C^{0,1}_{\log}(\mathbb{R}^d)$.

        The proof of $\Lambda^1(\mathbb{R}^d)\subset C^{0,1}_{\log}(\mathbb{R}^d)$ is provided in \cite[Proposition 2.107]{bahouri2011fourier}, but the strictness is not.
For the sake of completeness, we provide the proof and a counterexample.
    For every $m\in\mathbb{N}$,
    \begin{equation}
    \label{25.10.04.15.19}
    \mathcal{D}_hf=-\frac{1}{2}\sum_{k=0}^{m-1}2^{-k}\mathcal{D}^2_{2^kh}f+2^{-m}\mathcal{D}_{2^mh}f.
    \end{equation}
    Taking $L^\infty$-norms gives
    \begin{align*}
    \|\mathcal{D}_hf\|_{L_{\infty}(\mathbb{R}^d)}&\leq \frac{1}{2}\sum_{k=0}^{m-1}2^{-k}\|\mathcal{D}_{2^kh}^2f\|_{L_{\infty}(\mathbb{R}^d)}+2^{-m}\|\mathcal{D}_{2^mh}f\|_{L_{\infty}(\mathbb{R}^d)}\\
    &\leq \frac{1}{2}\sum_{k=0}^{m-1}2^{-k}(\|f\|_{\dot{\Lambda}^1(\mathbb{R}^d)}2^k|h|)+2^{1-m}\|f\|_{L_{\infty}(\mathbb{R}^d)}\\
    &=\frac{m|h|}{2}\|f\|_{\dot{\Lambda}^1(\mathbb{R}^d)}+2^{1-m}\|f\|_{L_{\infty}(\mathbb{R}^d)}.
    \end{align*}
    For $|h|<1$, choose $m=\lfloor\log_2(2/|h|)\rfloor$ so that $m\leq \log_2(2/|h|)$ and $2^{-m}\leq |h|$.
    Then
    $$
    \|\mathcal{D}_hf\|_{L_{\infty}(\mathbb{R}^d)}\leq |h|\log_2\left(\frac{2}{|h|}\right)\|f\|_{\dot{\Lambda}^1(\mathbb{R}^d)}+2|h|\|f\|_{L_{\infty}(\mathbb{R}^d)}\leq 3\|f\|_{\Lambda^1(\mathbb{R}^d)}|h|\log_2\left(\frac{2}{|h|}\right),
    $$
    since $\log_2(2/|h|)\geq1$ for $|h|<1$.
    Taking the supremum over $|h|<1$ yields
   \begin{equation}
    \label{25.11.12.23.22}
    \sup_{|x-y|<1}\frac{|f(x)-f(y)|}{|x-y|\log_2(2/|x-y|)}\leq 2\|f\|_{\Lambda^1(\mathbb{R}^d)}.
    \end{equation}

   To show $\Lambda^1(\mathbb{R}^d)\neq C^{0,1}_{\log}(\mathbb{R}^d)$,
   we put $g(x):=|x|\log_2(2/|x|)$.
   Then
   $$
   |g(h)-g(-h)-2g(0)|=2|h|\log_2(2/|h|),
   $$
   which implies that $g\zeta\not\in\Lambda^1(\mathbb{R}^d)$.
   However, $g\zeta\in C^{0,1}_{\log}(\mathbb{R}^d)$.
    The proposition is proved.
\end{proof}

\subsection{Properties of H\"older-Zygmund Spaces}
We next record several lemmas used throughout the proof.  
They provide 
\begin{itemize}
    \item a tool for localization via partition of unity (Lemma \ref{25.10.25.19.22}),
    \item an interpolation between Zygmund spaces (Lemma \ref{25.09.26.14.09}),
    \item estimates for products of functions in Zygmund classes (Lemma \ref{25.09.26.12.44}).
\end{itemize}
All of these will be employed in Step 1 of the proof of Theorem \ref{25.10.25.16.21}.

The following partition of unity lemma is a modified version of \cite[Lemma 4.1.1]{K96}, and we include the proof for completeness.
\begin{lem}[Partition of unity]
\label{25.10.25.19.22}
Let $\gamma>0$, and let $\zeta\in C_c^\infty(\mathbb{R}^d)$ be a cut-off function such that 
$\zeta(x)=1$ for $|x|\leq 1$, $\zeta(x)=0$ for $|x|\geq 2$, and $0\leq \zeta\leq 1$. 
Then for any $R>0$, $y\in\mathbb{R}^d$, and $u\in \Lambda^{\gamma}(\mathbb{R}^d)$, we have
\begin{equation}\label{eq9151537}
    \|u\|_{\Lambda^{\gamma}(\mathbb{R}^d)}\leq N_0\sup_{z\in\mathbb{R}^d} \|u\,\zeta_R^z\|_{\dot{\Lambda}^{\gamma}(\mathbb{R}^d)}+N_1\sup_{z\in\mathbb{R}^d}\|u\zeta_R^z\|_{L_{\infty}(\mathbb{R}^d)},
\end{equation}
where $\zeta_R^z(x):=\zeta\left(\frac{x-z}{R}\right)$, $N_0=N_0(d,\gamma,\zeta)$, and $N_1=N_1(d,\gamma,R,\zeta)$.
\end{lem}
\begin{proof}
We first claim that
\begin{equation} \label{eq9151546}
\|u\|_{\dot{\Lambda}^{\gamma}(\mathbb{R}^d)} \leq N(\gamma,R)\|u\|_{ L_{\infty}(\mathbb{R}^d)}+ \sup_{y\in\mathbb{R}^d}
\sup_{|h|\leq R/[\gamma]^{-}}\frac{\|\mathcal{D}_h^{[\gamma]^{-}}u\|_{L_{\infty}(B_R(y))}}{|h|^{\gamma}}.
\end{equation}
Indeed, for any $h,y\in\mathbb{R}^d$ and $x\in B_R(y)$,
\begin{align*}
\frac{|\mathcal{D}_h^{[\gamma]^{-}}u(x)|}{|h|^{\gamma}} =& \frac{|\mathcal{D}_h^{[\gamma]^{-}}u(x)|}{|h|^{\gamma}} 1_{|h|\leq \frac{R}{[\gamma]^{-}}} + \frac{|\mathcal{D}_h^{[\gamma]^{-}}u(x)|}{|h|^{\gamma}} 1_{|h|>\frac{R}{[\gamma]^{-}}}
\\
\leq& \sup_{|h|\leq \frac{R}{[\gamma]^{-}}}\frac{\|\mathcal{D}_h^{[\gamma]^{-}}u\|_{L_{\infty}(B_R(y))}}{|h|^{\gamma}} + N(\gamma,R)\|u\|_{L_{\infty}(\mathbb{R}^d)}.
\end{align*}

Next, fix $y\in\mathbb{R}^d$. 
For $x\in B_{2R}(y)$, we have the identity
\begin{equation}
\label{25.09.25.15.29}
    u(x) = cR^{-d} \int_{|z-y|\leq 4R} u(x)\zeta\left(\frac{x-z}{R}\right) \mathrm{d}z=cR^{-d}\int_{|z-y|\leq 4R} u(x)\zeta_R^z(x) \mathrm{d}z,
\end{equation}
where $c^{-1} := \int_{\mathbb{R}^d} \zeta(z) \mathrm{d}z$.
Due to \eqref{25.09.25.15.29},
\begin{equation}
\label{25.10.20.20.19}
\|u\|_{L_{\infty}(\mathbb{R}^d)}\leq N\sup_{z\in\mathbb{R}^d}\|u\zeta_R^z\|_{L_{\infty}(\mathbb{R}^d)}.
\end{equation}
For any natural number $k \leq [\gamma]^{-}$ and $|h| \leq R / [\gamma]^{-}$, we have for $
x \in B_R(y)$,  $x + k h \in B_{2R}(y)$.
Hence, by \eqref{25.09.25.15.29},
\begin{align*}
u(x + k h)
&= c R^{-d} \int_{|z - y| \le 4R} u(x + k h)\, \zeta\!\left(\frac{x + k h - z}{R}\right) \mathrm{d}z \\
&= c R^{-d} \int_{|z - y| \le 4R} u(x + k h)\, \zeta_R^z(x + k h)\, \mathrm{d}z.
\end{align*}
Therefore,
$$
\mathcal{D}_h^{[\gamma]^{-}} u(x)
= c R^{-d} \int_{|z - y| \le 4R} 
\mathcal{D}_h^{[\gamma]^{-}}\big(u \zeta_R^z\big)(x)\, \mathrm{d}z,
\qquad \forall\, x \in B_R(y).
$$
Taking the supremum over $x\in B_R(y)$ yields
\begin{equation}
\label{25.10.20.20.20}
\begin{aligned}
    \frac{\|\mathcal{D}_h^{[\gamma]^{-}}u\|_{L_{\infty}(B_R(y))}}{|h|^{\gamma}} &\leq cR^{-d} \int_{|z-y|\leq 4R} 
\frac{\|\mathcal{D}_h^{[\gamma]^{-}}(u\zeta_R^z)\|_{L_{\infty}(\mathbb{R}^d)}}{|h|^{\gamma}}\,\mathrm{d}z\leq N \|u\zeta_R^z\|_{\dot{\Lambda}^{\gamma}_{\infty}(\mathbb{R}^d)}.
\end{aligned}
\end{equation}
Combining \eqref{eq9151546}, \eqref{25.10.20.20.19}, and \eqref{25.10.20.20.20} completes the proof.
\end{proof}

\begin{lem}[Interpolation inequality]
\label{25.09.26.14.09}
    Let $\gamma\in(0,\infty)$ and $f\in \Lambda^{\gamma}(\mathbb{R}^d)$.
    Then for $\theta\in(0,1)$ and $\varepsilon>0$,
    $$
    \|f\|_{\dot{\Lambda}^{\theta\gamma}(\mathbb{R}^d)}\leq N\varepsilon\|f\|_{\dot{\Lambda}^{\gamma}(\mathbb{R}^d)}+N\varepsilon^{-\frac{\theta}{1-\theta}}\|f\|_{L_{\infty}(\mathbb{R}^d)},\quad \forall \varepsilon>0,
    $$
    where $N$ is independent of $f$ and $\varepsilon$.
\end{lem}
\begin{proof}
    By \cite[Theorem 6.4.5]{bergh2012interpolation} and Proposition \ref{25.10.20.14.01},

    $$
    \Lambda^{\theta\gamma}(\mathbb{R}^d)=(B_{\infty,\infty}^{0}(\mathbb{R}^d),\Lambda^{\gamma}(\mathbb{R}^d))_{\theta,\infty},
    $$
    where $B_{\infty,\infty}^{0}(\mathbb{R}^d)$ denotes the Besov space of order $0$.
    By \cite[(6) in Section 2.4.1]{MR781540},
    \begin{equation}
    \label{25.10.22.17.19}
    \begin{aligned}
    \|f\|_{\Lambda^{\theta\gamma}(\mathbb{R}^d)}\simeq\|f\|_{(B_{\infty,\infty}^{0}(\mathbb{R}^d),\Lambda^{\gamma}(\mathbb{R}^d))_{\theta,\infty}}\leq 2\|f\|_{B_{\infty,\infty}^{0}(\mathbb{R}^d)}^{1-\theta}\|f\|_{\Lambda^{\gamma}(\mathbb{R}^d)}^{\theta}.
    \end{aligned}
    \end{equation}
Let
$$
S_0 f(x)
= \int_{\mathbb{R}^d} f(x-y)\Phi(y)\,\mathrm{d}y,
\qquad
\Delta_j f(x)
= \int_{\mathbb{R}^d} f(x-y)\psi_j(y)\,\mathrm{d}y,
$$
where $\Phi := \sum_{j\leq 0}\psi_j$.
Then we obtain
\begin{equation}
\label{25.12.06.14.07}
\|f\|_{B_{\infty,\infty}^0(\mathbb{R}^d)}:=\|S_0f\|_{L_{\infty}(\mathbb{R}^d)}+\sup_{j\in\mathbb{N}}\|\Delta_jf\|_{L_{\infty}(\mathbb{R}^d)}
\leq C_{\psi}\|f\|_{L_{\infty}(\mathbb{R}^d)},
\end{equation}
where
$$
C_{\psi}
:= \|\Phi\|_{L^1(\mathbb{R}^d)}
   + \sup_{j\in\mathbb{Z}}\|\psi_j\|_{L^1(\mathbb{R}^d)}
<\infty.
$$
    For any $c,\eta>0$,
    \begin{equation}
    \label{25.10.22.17.21}
    \|f(c\cdot)\|_{\dot{\Lambda}^{\eta}(\mathbb{R}^d)}=c^{\eta}\|f\|_{\dot{\Lambda}^{\eta}(\mathbb{R}^d)}.
    \end{equation}
    Combining \eqref{25.10.22.17.19}–\eqref{25.10.22.17.21} yields
    \begin{equation}
    \begin{aligned}
    \|f(c\cdot)\|_{\Lambda^{\theta\gamma}(\mathbb{R}^d)}&=\|f\|_{L_{\infty}(\mathbb{R}^d)}+c^{\theta\gamma}\|f\|_{\dot{\Lambda}^{\theta\gamma}(\mathbb{R}^d)}\\
    &\leq N\|f\|_{L_{\infty}(\mathbb{R}^d)}^{1-\theta}(\|f\|_{L_{\infty}(\mathbb{R}^d)}+c^{\gamma}\|f\|_{\dot{\Lambda}^{\gamma}(\mathbb{R}^d)})^{\theta}\\
    &\leq N\|f\|_{L_{\infty}(\mathbb{R}^d)}+Nc^{\theta\gamma}\|f\|_{L_{\infty}(\mathbb{R}^d)}^{1-\theta}\|f\|_{\dot{\Lambda}^{\gamma}(\mathbb{R}^d)}^{\theta},
    \end{aligned}
    \end{equation}
    where $N$ is independent of $c$.
    Dividing both sides by $c^{\theta\gamma}$ and letting $c \to \infty$, we obtain
    $$
    \|f\|_{\dot{\Lambda}^{\theta\gamma}(\mathbb{R}^d)}\leq N\|f\|_{L_{\infty}(\mathbb{R}^d)}^{1-\theta}\|f\|_{\dot{\Lambda}^{\gamma}(\mathbb{R}^d)}^{\theta}.
    $$
    Applying Young’s inequality gives the desired estimate.
\end{proof}

\begin{lem}
\label{25.09.26.12.44}
    Let $\gamma\in(0,\infty)$ and $f,g\in \Lambda^{\gamma}(\mathbb{R}^d)$.
    Then $fg\in \Lambda^{\gamma}(\mathbb{R}^d)$, and
    \begin{equation}
    \label{25.09.26.12.48}
    \begin{aligned}
    \|fg\|_{\dot{\Lambda}^{\gamma}(\mathbb{R}^d)}\leq& \|f\|_{L_{\infty}(\mathbb{R}^d)}\|g\|_{\dot{\Lambda}^{\gamma}(\mathbb{R}^d)}+\sum_{j=1}^{[\gamma]^{-}-1}\binom{[\gamma]^{-}}{j}\|f\|_{\dot{\Lambda}^{j}(\mathbb{R}^d)}\|g\|_{\dot{\Lambda}^{\gamma-j}(\mathbb{R}^d)}\\
    &+\|f\|_{\dot{\Lambda}^{\gamma}(\mathbb{R}^d)}\|g\|_{L_{\infty}(\mathbb{R}^d)}.
    \end{aligned}
    \end{equation}
\end{lem}
\begin{proof}
The inequality \eqref{25.09.26.12.48} can be directly obtained from
$$
\mathcal{D}_h^{[\gamma]^{-}}(fg)(x)
=
\sum_{j=0}^{[\gamma]^{-}}
\binom{[\gamma]^{-}}{j}
\mathcal{D}_h^{j} f(x + ([\gamma]^{-} - j)h)
\, \mathcal{D}_h^{[\gamma]^{-} - j} g(x).
$$
The lemma is proved.
\end{proof}

\section{Proof of Theorem \ref{25.10.20.19.28}}
\label{25.12.06.14.56}

\subsection{Proof of Theorem \ref{25.10.20.19.28} with zero initial data}
In this subsection, we consider the Cauchy problem
\begin{equation} \label{main_loc}
    \partial_t u (t,x)
    = \sum_{i,j=1}^{d} a_{ij}(t,x) D_{ij}u(t,x)
      + \sum_{i=1}^{d} b_i(t,x) D_i u
      + c(t,x)u(t,x) + f(t,x),
\end{equation}
subject to the zero initial condition $u(0,\cdot)=0$.

\begin{thm}
\label{25.10.25.16.21}
Let $T \in (0,\infty)$, $p \in (1,\infty]$, $\gamma > 0$, and let $w \in A_p(\mathbb{R})$ with $[w]_{A_p}\leq K_0$ (put $w\equiv1$ and $K_0=1$ when $p=\infty$).
Suppose that Assumption \ref{25.10.12.23.59} ($\gamma$) holds.  
Then, for any $f \in  \mathbf{\Lambda}_{p,w}^{\gamma}(T)$, there exists a unique solution 
$u \in \mathbf{H}_{p,w}^{\gamma+2}(T)$ to the equation \eqref{main_loc} in $(0,T) \times \mathbb{R}^d$, with the zero initial condition $u(0,x) = 0$.
Moreover, the following estimate holds:
$$
    \|u\|_{\mathbf{H}_{p,w}^{\gamma+2}(T)}
    \leq N\|f\|_{\mathbf{\Lambda}_{p,w}^{\gamma}(T)},
$$
where $N = N(d,\gamma,K_0,K,\nu,T,p)$.
\end{thm}

The next two lemmas play crucial roles in the proof of the a priori estimate.  
\begin{lem}
\label{25.10.25.20.19}
    Let $\eta>0$, $p\in(1,\infty]$, and $T>0$.
    Suppose that
    \begin{itemize}
        \item $a=(a_{ij}(t))_{i,j}$ is independent of $x$.
        \item $a=(a_{ij}(t))_{i,j}$ is bounded and satisfies \eqref{25.10.20.19.33}.
        \item $b_i=c=0$ for all $i=1,\cdots,d$.
    \end{itemize}
If $u$ is a solution to \eqref{main_loc} with $f\in \mathbf{\Lambda}_{p,w}^{\eta}(T)$, then for each $t\in(0,T)$,
    \begin{equation} \label{const_loc}
        \|u(t,\cdot)\|_{\Lambda^{\eta+2}(\mathbb{R}^d)} \leq N(T) \cM_t\left(\|f(*,\cdot)\|_{\Lambda^{\eta}(\mathbb{R}^d)} 1_{(0,T)}(*)\right)(t),
    \end{equation}
where
$$
\mathcal{M}_tf(t):=\sup_{r>0}\aint_{t-r}^{t+r}f(s)\mathrm{d}s.
$$
In particular, $N(T)$ is increasing in $T$.
\end{lem}
\begin{proof}
    This is a direct consequence of \cite[Lemma 2.6]{C24} with $\phi(\lambda)=\lambda^{\eta}$, $\gamma=2$, and $\phi_{\gamma}(\lambda)=\lambda^{\eta+2}$.
\end{proof}

\begin{lem}
\label{25.10.25.21.37}
Let $\gamma\in(0,\infty)$, $p\in(1,\infty]$, and $w\in A_p$ such that $[w]_{A_p}\leq K_0$ (put $w\equiv1$ and $K_0=1$ when $p=\infty$).
Suppose that $u \in \mathbf{\Lambda}_{p,w}^{\gamma}(a,b)$ satisfies $u(a,\cdot) = 0$.
Then we have
    \begin{equation} \label{eq9162040}
        \|u\|_{L_{p}((a,b),w\,\mathrm{d}t;L_{\infty}(\mathbb{R}^d))} \leq N(b-a)\|\partial_tu\|_{L_{p}((a,b),w\,\mathrm{d}t;L_{\infty}(\mathbb{R}^d))},
    \end{equation}
    where $N=N(d,p,K_0)$.
\end{lem}
\begin{proof}
    By the fundamental theorem of calculus,
    \begin{align*}
        |u(t,x)| \leq \int_a^t |\partial_t u(s,x)| \mathrm{d}s = (t-a)\aint_a^t |\partial_t u(s,x)| \mathrm{d}s.
    \end{align*}
    Consequently,
    \begin{equation*}
        \|u(t,\cdot)\|_{L_{\infty}(\mathbb{R}^d)} \leq 2(t-a)\cM_t(\|\partial_tu(\ast,\cdot)\|_{L_{\infty}(\mathbb{R}^d)}1_{(a,b)}(\ast))(t),\quad \forall t\in(a,b).
    \end{equation*} 
Applying the weighted Hardy--Littlewood maximal function theorem (see, \textit{e.g.}, \cite[Theorem 2.2]{DK18}) yields \eqref{eq9162040}.  
This completes the proof.
\end{proof}

We now prove the a priori estimate for the case with zero initial data.
\begin{lem}[A priori estimate]
\label{25.10.30.10.49}
    Let $\gamma\in(0,\infty)$, $p\in(1,\infty]$, $T>0$, and $w\in A_p$ with $[w]_{A_p}\leq K_0$ (put $w\equiv1$ and $K_0=1$ when $p=\infty$). 
    Suppose that $f\in \mathbf{\Lambda}_{p,w}^{\gamma}(T)$, and $u\in \mathbf{H}_{p,w}^{\gamma+2}(T)$ is a solution to \eqref{main_loc}. 
    Then
    \begin{equation} \label{eq9161218}
        \|\partial_tu\|_{\mathbf{\Lambda}_{p,w}^{\gamma}(T)} + \|u\|_{\mathbf{\Lambda}_{p,w}^{\gamma+2}(T)} \leq N  \|f\|_{\mathbf{\Lambda}_{p,w}^{\gamma}(T)},
    \end{equation}
    where $N=N(d,\nu,K,p,K_0,T)$.
\end{lem}
\begin{proof}
We divide the proof into two steps.

\textbf{Step 1.} We first establish
\begin{equation} \label{25.10.30.10.17}
        \|\partial_tu\|_{\mathbf{\Lambda}_{p,w}^{\gamma}(T)} + \|u\|_{\mathbf{\Lambda}_{p,w}^{\gamma+2}(T)} \leq N\left(  \|f\|_{\mathbf{\Lambda}_{p,w}^{\gamma}(T)} + \|u\|_{L_p((0,T),w\,\mathrm{d}t;L_{\infty}(\mathbb{R}^d))}\right),
    \end{equation}
    where $N=N(d,\nu,K,K_0,p,T)$.

Since $u$ satisfies the equation \eqref{main_loc}, we have
\begin{align*}
\|\partial_t u\|_{\mathbf{\Lambda}_{p,w}^{\gamma}(T)}
=&
\left\|\sum_{i,j=1}^{d}a_{ij} D_{ij}u + \sum_{i=1}^{d}b_i D_i u + c u + f\right\|_{\mathbf{\Lambda}_{p,w}^{\gamma}(T)}\\
\leq&
N\left(\|u\|_{\mathbf{\Lambda}_{p,w}^{\gamma+2}(T)} + \|f\|_{\mathbf{\Lambda}_{p,w}^{\gamma}(T)}\right).
\end{align*}
Thus, in order to obtain \eqref{25.10.30.10.17}, it remains  to estimate $\|u\|_{\mathbf{\Lambda}_{p,w}^{\gamma+2}(T)}$ in terms of the right-hand side.

To achieve this, it is enough to deal with the case $b_i \equiv c \equiv 0$.  
Indeed, once the required estimate holds under this assumption, we have
\begin{align*}
\|\partial_t u\|_{\mathbf{\Lambda}_{p,w}^{\gamma}(T)}
+
\|u\|_{\mathbf{\Lambda}_{p,w}^{\gamma+2}(T)}
\leq&
N\left\|f + \sum_{i=1}^{d}b_i D_i u + c u\right\|_{\mathbf{\Lambda}_{p,w}^{\gamma}(T)}\\
\leq&
N\left(\|f\|_{\mathbf{\Lambda}_{p,w}^{\gamma}(T)}
+
\|u\|_{\mathbf{\Lambda}_{p,w}^{\gamma+1}(T)}\right).
\end{align*}
We now estimate the remaining $\|u\|_{\mathbf{\Lambda}_{p,w}^{\gamma+1}(T)}$ term.  
By Lemma \ref{25.09.26.14.09},
\begin{equation}
\label{25.11.13.11.35}
\|u\|_{\mathbf{\Lambda}_{p,w}^{\gamma+1}(T)}
\le
N\varepsilon \|u\|_{\mathbf{\Lambda}_{p,w}^{\gamma+2}(T)}
+
N(\varepsilon)\,
\|u\|_{L_p((0,T), w\,\mathrm{d}t; L_{\infty}(\mathbb{R}^d))}.
\end{equation}
Choosing $\varepsilon > 0$ sufficiently small allows the first term on the right-hand side to be absorbed into the left-hand side of the estimate, thereby yielding \eqref{25.10.30.10.17}.
Therefore, we assume that $f\in \mathbf{\Lambda}_{p,w}^{\gamma}(T)$, and $u\in \mathbf{H}_{p,w}^{\gamma+2}(T)$ is a solution to \eqref{main_loc} with $b_i=c=0$ and the zero initial condition.

We first consider $\gamma\in(0,1]$.
Let $\zeta\in C_c^\infty(\mathbb{R}^d)$ such that $\zeta=1$ on $B_1$, $\zeta=0$ on $\mathbb{R}^d\setminus B_2$, and $0\leq\zeta\leq1$. 
For $R>0$ and $y\in\mathbb{R}^d$, define $\zeta_R^y(x):=\zeta((x-y)/R)$ and set $v^y:=u\zeta_R^y$. 
Then $v^y$ satisfies
\begin{align*}
    \partial_tv^{y} &= \sum_{i,j=1}^{d}a_{ij}D_{ij}v^{y} + \zeta_R^y f - \sum_{i,j=1}^{d}(a_{ij}+a_{ji})D_iu D_j\zeta_R^y - \sum_{i,j=1}^{d}a_{ij}uD_{ij}\zeta_R^y
    \\
    &= \sum_{i,j=1}^{d}a_{ij}(t,y)D_{ij}v^{y} + F,
\end{align*}
where
\begin{align*}
    F &:= \sum_{i,j=1}^{d}(a_{ij}-a_{ij}(t,y))D_{ij}v^{y} +  \zeta_R^y f - \sum_{i,j=1}^{d}(a_{ij}+a_{ji})D_iu D_j\zeta_R^y - \sum_{i,j=1}^{d}a_{ij}uD_{ij}\zeta_R^y 
    \\
    &=: \sum_{i,j=1}^{d}(a_{ij}-a_{ij}(t,y))D_{ij}v^{y} + G.
\end{align*}
Due to $\Lambda^\gamma = C^\gamma$ for $\gamma\in(0,1)$, and \eqref{25.11.12.23.22}, $\|a_{ij}(t,\cdot)-a_{ij}(t,y)\|_{L_{\infty}(B_{2R}(y))} \leq NA_{\gamma}(R)$, where $N=N(\nu)$ and $A_\gamma(R):=R^{\gamma}1_{\gamma\in(0,1)}+R\log_2(2/R)1_{\gamma=1}$.
Using Lemma \ref{25.09.26.12.44}, Proposition \ref{25.10.20.14.01}-($ii$) and Lemma \ref{25.09.26.14.09} with $\varepsilon=A_{\gamma}(R)$, we have
\begin{equation}
\label{eq9151852}
\begin{aligned}
    &\left\|\sum_{i,j=1}^{d}(a_{ij}(t,\cdot)-a_{ij}(t,y))D_{ij}v^{y}(t,\cdot)\right\|_{\Lambda^{\gamma}(\mathbb{R}^d)} \\
    \leq& N(d)\sum_{i,j=1}^{d}\|(a_{ij}(t,\cdot)-a_{ij}(t,y))\|_{L_{\infty}(B_{2R}(y))} \|D_x^2 v^{y}(t,\cdot)\|_{\Lambda^{\gamma}(\mathbb{R}^d)}
    \\
    \quad &+ N(d)\sum_{i,j=1}^{d}\|(a_{ij}(t,\cdot)-a_{ij}(t,y))\|_{\Lambda^{\gamma}(\mathbb{R}^d)} \|D_x^2 v^{y}(t,\cdot)\|_{L_{\infty}(\mathbb{R}^d)}
    \\
    \leq& N A_{\gamma}(R) \|v^{y}(t,\cdot)\|_{\Lambda^{\gamma+2}(\mathbb{R}^d)} + N\|v^{y}(t,\cdot)\|_{\Lambda^2(\mathbb{R}^d)} 
    \\
    \leq& N A_{\gamma}(R) \|v^{y}(t,\cdot)\|_{\Lambda^{\gamma+2}(\mathbb{R}^d)} + N(d,\nu,\gamma,R)\|u(t,\cdot)\|_{L_{\infty}(\mathbb{R}^d)}.
\end{aligned}
\end{equation}
Due to \eqref{25.10.22.17.21}, for $n\in\mathbb{N}\cup\{0\}$,
\begin{equation}
\label{25.11.13.11.24}
\|D_{x}^n\zeta_R^y\|_{\Lambda^{\gamma}(\mathbb{R}^d)}=R^{-n}\|D_{x}^n\zeta\|_{L_{\infty}(\mathbb{R}^d)}+R^{-n-\gamma}\|D_{x}^n\zeta\|_{\dot{\Lambda}^{\gamma}(\mathbb{R}^d)}.
\end{equation}
For $G$, Lemma \ref{25.09.26.12.44} and \eqref{25.11.13.11.24} imply
\begin{equation}
\label{eq9151853}
\begin{aligned}
     \|G(t,\cdot)\|_{\Lambda^{\gamma}(\mathbb{R}^d)} \leq& N\|\zeta_R^y \|_{\Lambda^{\gamma}(\mathbb{R}^d)}\|f(t,\cdot)\|_{\Lambda^{\gamma}(\mathbb{R}^d)} \\
     &+ N(\|D_x\zeta_R^y \|_{\Lambda^{\gamma}(\mathbb{R}^d)}+\|D_{xx}\zeta_R^y \|_{\Lambda^{\gamma}(\mathbb{R}^d)}) \|u(t,\cdot)\|_{\Lambda^{\gamma+1}(\mathbb{R}^d)} \\
     \leq& N(R)\|f(t,\cdot)\|_{\Lambda^{\gamma}(\mathbb{R}^d)} + N(R)\|u(t,\cdot)\|_{\Lambda^{\gamma+1}(\mathbb{R}^d)},
\end{aligned}
\end{equation}
where $N(R)$ is independent of $y$.
Combining \eqref{eq9151852}, \eqref{eq9151853}, and Lemma \ref{25.10.25.20.19} gives
\begin{equation}
\label{25.12.27.12.11}
\begin{aligned}
    &\|v^{y}(t,\cdot)\|_{\Lambda^{\gamma+2}(\mathbb{R}^d)}\\
    \leq& N\cM_t\left(\|F(*,\cdot)\|_{\Lambda^{\gamma}(\mathbb{R}^d)} 1_{(0,T)}(*)\right)(t)
    \\
    \leq& N \cM_t\left(\|(a_{ij}-a_{ij}(*,y))D_{ij}v^{y}(*,\cdot)\|_{\Lambda^{\gamma}(\mathbb{R}^d)} 1_{(0,T)}(*)\right)(t)
    \\
    &+ N \cM_t\left(\|G(*,\cdot)\|_{\Lambda^{\gamma}(\mathbb{R}^d)} 1_{(0,T)}(*)\right)(t)
    \\
    \leq& NA_{\gamma}(R) \cM_t\left(\|v^y(\ast,\cdot)\|_{\Lambda^{\gamma+2}(\mathbb{R}^d)} 1_{(0,T)}(*)\right)(t)\\
    &+N(R)\cM_t\left((\|u(\ast,\cdot)\|_{\Lambda^{\gamma+1}(\mathbb{R}^d)}+\|f(\ast,\cdot)\|_{\Lambda^{\gamma}(\mathbb{R}^d)}+\|u(\ast,\cdot)\|_{L_{\infty}(\mathbb{R}^d)})1_{(0,T)}(\ast)\right)(t).
\end{aligned}
\end{equation}
Since the constants in \eqref{25.12.27.12.11} are independent of $y$, we have
\begin{align*}
&\sup_{y\in\mathbb{R}^d}\|v^{y}(t,\cdot)\|_{\Lambda^{\gamma+2}(\mathbb{R}^d)}\\
\leq& NA_{\gamma}(R)\mathcal{M}_t\left(\sup_{y\in\mathbb{R}^d}\|v^y(\ast,\cdot)\|_{\Lambda^{\gamma+2}(\mathbb{R}^d)} 1_{(0,T)}(*)\right)(t)\\
&+N(R)\cM_t\left((\|u(\ast,\cdot)\|_{\Lambda^{\gamma+1}(\mathbb{R}^d)}+\|f(\ast,\cdot)\|_{\Lambda^{\gamma}(\mathbb{R}^d)}+\|u(\ast,\cdot)\|_{L_{\infty}(\mathbb{R}^d)})1_{(0,T)}(\ast)\right)(t)
\end{align*}
By the Hardy-Littlewood maximal function theorem, when $p\in(1,\infty)$,
\begin{align*}
    &\int_{0}^T \sup_{y\in\mathbb{R}^d} \|v^{y}(t,\cdot)\|_{\Lambda^{\gamma+2}(\mathbb{R}^d)}^p w(t)\, \mathrm{d}t 
    \\
    \leq& NA_\gamma(R) \int_{0}^T \sup_{y\in\mathbb{R}^d} \|v^{y}(t,\cdot)\|_{\Lambda^{\gamma+2}(\mathbb{R}^d)}^p w(t)\, \mathrm{d}t
    \\
    &+ N(R)\bigg(\|u\|_{\mathbf{\Lambda}_{p,w}^{\gamma+1}(T)}+\|f\|_{\mathbf{\Lambda}_{p,w}^{\gamma}(T)}+\|u\|_{L_p((0,T),w\,\mathrm{d}t;L_{\infty}(\mathbb{R}^d))}\bigg).
\end{align*}
By taking a sufficiently small $R=R_0\in(0,1)$ such that $NA_{\gamma}(R_0)<1$ ($N$ is independent of $R$), we have
\begin{equation}
\label{eq9152333}
\begin{aligned}
&\int_{0}^T\sup_{y\in\mathbb{R}^d}\|v^y(t,\cdot)\|_{{\Lambda}^{\gamma+2}(\mathbb{R}^d)}w(t)\,\mathrm{d}t\\
\leq& N(R_0)\left(\|u\|_{\mathbf{\Lambda}_{p,w}^{\gamma+1}(T)}+\|f\|_{\mathbf{\Lambda}_{p,w}^{\gamma}(T)}+\|u\|_{L_p((0,T),w\,\mathrm{d}t;L_{\infty}(\mathbb{R}^d))}\right).
\end{aligned}
\end{equation}
Since by Lemma \ref{25.10.25.19.22},
\begin{align*}
    \|u\|_{\mathbf{\Lambda}_{p,w}^{\gamma+2}(T)}^p =& \int_{0}^T \|u(t,\cdot)\|_{\Lambda^{\gamma+2}(\mathbb{R}^d)}^p w(t) \,\mathrm{d}t \\
    \leq& N(R_0) \int_{0}^T \sup_{y\in\mathbb{R}^d} \|v^{y}(t,\cdot)\|_{\Lambda^{\gamma+2}(\mathbb{R}^d)}^p w(t)\, \mathrm{d}t
\end{align*}
one can deduce that
\begin{align} \label{25.12.28.21.04}
        \|u\|_{\mathbf{\Lambda}_{p,w}^{\gamma+2}(T)}\leq
        N(R_0) \left(\|u\|_{\mathbf{\Lambda}_{p,w}^{\gamma+1}(T)} +  \|f\|_{\mathbf{\Lambda}_{p,w}^{\gamma}(T)} + \|u\|_{L_{p}((0,T),w\,\mathrm{d}t;L_{\infty}(\mathbb{R}^d))} \right).
    \end{align}
when $p\in(1,\infty)$.
Similarly, when $p=\infty$, by using
\begin{align*}
    \|u\|_{\mathbf{\Lambda}_{p,w}^{\gamma+2}(T)} = \sup_{t\in(0,T)} \|u(t,\cdot)\|_{\Lambda^{\gamma+2}(\mathbb{R}^d)}  \leq N(R_0) \sup_{t\in(0,T)} \sup_{y\in\mathbb{R}^d} \|v^{y}(t,\cdot)\|_{\Lambda^{\gamma+2}(\mathbb{R}^d)},
\end{align*}
one can still obtain \eqref{25.12.28.21.04} with $p=\infty$ (and $w\equiv1$).
    Using \eqref{25.11.13.11.35} and taking a sufficiently small $\varepsilon\in(0,1)$ such that $N(R_0)\varepsilon<1$ yields the result for $\gamma\in(0,1]$.

The general case $\gamma\in(n-1,n]$ follows by induction using Proposition \ref{25.10.20.14.01}-(ii).
Assume that \eqref{25.10.30.10.17} holds when $\tilde{\gamma}\in (0,n-1]$ with $n\geq2$. 
Proposition \ref{25.10.20.14.01}-(ii) yields $\pi_{k}:=D_{k}u\in \mathbf{H}_{p,w}^{\gamma+1}(T)$ and $\pi_k$ satisfies
\begin{align*}
    \partial_t\pi_{k} &= \sum_{i,j=1}^{d}a_{ij}D_{ij}\pi_{k} + D_k f + \sum_{i,j=1}^{d}D_ka_{ij}D_{ij}u.
\end{align*}
By the induction hypothesis with $\gamma-1\in(0,n-1]$, Proposition \ref{25.10.20.14.01}-(ii) and Lemma \ref{25.09.26.12.44},
\begin{align*}
    \|\pi_{k}\|_{\mathbf{\Lambda}_{p,w}^{\gamma+1}(T)} &\leq N \|\pi_{k}\|_{L_p((0,T),w\,\mathrm{d}t;L_{\infty}(\mathbb{R}^d))} + N\|D_k f\|_{\mathbf{\Lambda}_{p,w}^{\gamma-1}(T)}
    + N \sum_{i,j=1}^{d}\|D_ka_{ij}D_{ij}u\|_{\mathbf{\Lambda}_{p,w}^{\gamma-1}(T)}
\\
&\leq N\left(\|u\|_{\mathbf{\Lambda}_{p,w}^{\gamma+1}(T)} + N \|f\|_{\mathbf{\Lambda}_{p,w}^{\gamma}(T)}\right).
\end{align*}
Since \eqref{25.10.30.10.17} is true when $\gamma-1\in(0,n-1]$,
\begin{equation*}
    \|u\|_{\mathbf{\Lambda}_{p,w}^{\gamma+1}(T)} \leq N  \left(  \|f\|_{\mathbf{\Lambda}_{p,w}^{\gamma-1}(T)} + \|u\|_{L_p((0,T),w\,\mathrm{d}t;L_{\infty}(\mathbb{R}^d))}\right).
\end{equation*}
Thus, we arrive at \eqref{25.10.30.10.17} with $\gamma\in(0,\infty)$.

\medskip

\textbf{Step 2.} We now show
\begin{equation*}
    \|u\|_{L_p((0,T),w\,\mathrm{d}t;L_{\infty}(\mathbb{R}^d))} \leq N \|f\|_{\mathbf{\Lambda}_{p,w}^{\gamma}(T)},
\end{equation*}    
where $N=N(d,\nu,K,p,K_0,T)$.
Let $s_l := \frac{lT}{m}$ for $l=0,\dots,m$, and take cut-off functions $\eta_l\in C^\infty(\mathbb{R})$ ($l\geq1$) such that $\eta_l=1$ for $t>s_l$, $\eta_l=0$ for $t\leq s_{l-1}$, and $|\eta_l'|\leq 2m/T$. 
     For $l=0$, set $s_{-1}:=0$ and $\eta_0\equiv 1$.
     Then $q^l(t,x):=u(t,x)\eta_l(t) \in \mathbf{\Lambda}_{p,w}^{\gamma+2}(s_{l-1},s_{l+1})$ satisfies
    \begin{equation*}
        \partial_tq^l = \sum_{i,j=1}^{d}a_{ij}D_{ij} q^l  + \sum_{i=1}^{d}b_i D_iq^l + cq^l + f\eta_l +u\partial_t \eta_l, \quad t\in(s_{l-1},s_{l+1}),
    \end{equation*}
    with the zero initial condition $q^l(s_{l-1},\cdot)=0$. 
    By \eqref{25.10.30.10.17} from Step 1, for $l\ge1$,
    \begin{align*}
        &\|\partial_tq^l\|_{\mathbf{\Lambda}_{p,w}^{\gamma}(s_{l-1},s_{l+1})}+\|q^l\|_{\mathbf{\Lambda}_{p,w}^{\gamma+2}(s_{l-1},s_{l+1})} \\
        \leq& N\|f\eta_l\|_{\mathbf{\Lambda}_{p,w}^{\gamma}(s_{l-1},s_{l+1})} + N\|u\partial_t\eta_l\|_{\mathbf{\Lambda}_{p,w}^{\gamma}(s_{l-1},s_{l+1})}+ N\|q^l\|_{L_p((s_{l-1},s_{l+1}),w\,\mathrm{d}t;L_{\infty}(\mathbb{R}^d))}
        \\
        \leq& N\|f\|_{\mathbf{\Lambda}_{p,w}^{\gamma}(T)} + N\frac{m}{T} \|u\|_{\mathbf{\Lambda}_{p,w}^{\gamma}(s_{l})}+ N\|u\|_{L_p((s_{l-1},s_{l}),w\,\mathrm{d}t;L_{\infty}(\mathbb{R}^d))}\\
        &+ N\|u\|_{L_p((s_{l},s_{l+1}),w\,\mathrm{d}t;L_{\infty}(\mathbb{R}^d))}.
    \end{align*}
    For $l=0$, since $\partial_t\eta_0=0$,
     \begin{align*}
        \|u\|_{\mathbf{\Lambda}_{p,w}^{\gamma+2}(s_1)} &\leq N\|f\|_{\mathbf{\Lambda}_{p,w}^{\gamma}(T)}+ N\|u\|_{L_{p}((0,s_1),w\,\mathrm{d}t;L_{\infty}(\mathbb{R}^d))}.
    \end{align*}
    Here, we remark that, although the analysis is carried out on smaller time intervals, the constants $N$ can be regarded as depending on $T$.
    Now apply Lemma \ref{25.10.25.21.37} and Step 1 to get, for $l\ge1$,
    \begin{align*}
        \|u\|_{L_{p}((s_l,s_{l+1}), w\,\mathrm{d}t;L_{\infty}(\mathbb{R}^d))} \leq& \|q^l\|_{L_{p}((s_{l-1},s_{l+1}), w\,\mathrm{d}t;L_{\infty}(\mathbb{R}^d))} \\
        \leq& N\frac{T}{m}\|\partial_t q^l\|_{L_{p}((s_{l-1},s_{l+1}), w\,\mathrm{d}t;L_{\infty}(\mathbb{R}^d))}
        \\
        \leq& N\frac{T}{m}\|f\|_{\mathbf{\Lambda}_{p,w}^{\gamma}(T)} + N\frac{T}{m}\|u\|_{L_p((s_{l-1},s_{l}),w\,\mathrm{d}t;L_{\infty}(\mathbb{R}^d))}
        \\
        & + N\frac{T}{m}\|u\|_{L_p((s_{l},s_{l+1}),w\,\mathrm{d}t;L_{\infty}(\mathbb{R}^d))}+N\|u\|_{\mathbf{\Lambda}_{p,w}^{\gamma}(s_{l})},
    \end{align*}
    and for $l=0$,
    \begin{align*}
        \|u\|_{L_p((0,s_1),w\,\mathrm{d}t;L_{\infty}(\mathbb{R}^d))} &\leq N\frac{T}{m}\left(\|f\|_{\mathbf{\Lambda}_{p,w}^{\gamma}(T)} + \|u\|_{L_p((0,s_1),w\,\mathrm{d}t;L_{\infty}(\mathbb{R}^d))} \right).
    \end{align*}
    Choose $m\in\mathbb N$ sufficiently large so that $N\,T/m<1$.
    Then, for every $l=1,\dots,m-1$,
    \begin{align*}
        \|u\|_{L_{p}((s_l,s_{l+1}),w\,\mathrm{d}t;L_{\infty}(\mathbb{R}^d))}\leq& N\left(\|f\|_{\mathbf{\Lambda}_{p,w}^{\gamma}(T)} + \|u\|_{\mathbf{\Lambda}_{p,w}^{\gamma}(s_{l})}\right)\\
        \leq& N\left(\|f\|_{\mathbf{\Lambda}_{p,w}^{\gamma}(T)} + \|u\|_{L_{p}((0,s_l),w\,\mathrm{d}t;L_{\infty}(\mathbb{R}^d))}\right).
    \end{align*}
    For $l=0$, the same argument yields
    \begin{equation*}
        \|u\|_{L_{p}((0,s_{1}),w\,\mathrm{d}t;L_{\infty}(\mathbb{R}^d))} \leq N\|f\|_{\mathbf{\Lambda}_{p,w}^{\gamma}(T)}.
    \end{equation*}
    Hence, by induction, the desired result is obtained.
Combining Steps 1 and 2, we complete the proof of the lemma.
\end{proof}

\medskip

\begin{proof}[Proof of Theorem \ref{25.10.25.16.21}]
    The uniqueness and the estimate follow directly from Lemma \ref{25.10.30.10.49}.

    For the existence, we apply the method of continuity.
    For $\lambda\in[0,1]$, define
    $$
    \mathcal{L}^{\lambda}:=\sum_{i,j=1}^{d}(\lambda a_{ij}+(1-\lambda)\delta_{ij})D_{ij}+\lambda \sum_{i=1}^{d}b_i D_i+\lambda c.
    $$
    Let $\mathrm{S}$ denote the set of all $\lambda \in [0,1]$ such that for every $f \in \mathbf{\Lambda}_{p,w}^{\gamma}(T)$, there exists a unique $u \in \mathbf{\Lambda}_{p,w}^{\gamma+2}(T)$ satisfying
    \begin{equation}
    \label{25.11.01.12.50}
\partial_t u = \mathcal{L}^{\lambda} u + f
\quad \text{with the zero initial condition.}
    \end{equation}

    We claim that $\mathrm{S} = [0,1]$.

    \textbf{Step 1. $\mathrm{S}$ is nonempty.}
    By known results (see, \textit{e.g.}, \cite[Theorem 1.6]{C24}), the equation corresponding to $\lambda = 0$ is solvable; hence $0 \in \mathrm{S}$.
    
\textbf{Step 2. Uniform openness of $\mathrm{S}$.}
Assumption~\ref{25.10.12.23.59}$(\gamma)$ holds for $\mathcal{L}^{\lambda}$ for
all $\lambda\in[0,1]$ with the same ellipticity constant $\nu$ and the same bound $K$ on the coefficients. 
Thus Lemma \ref{25.10.30.10.49} yields the a priori
estimate
\begin{equation}\label{25.12.11.15.57}
    \|u\|_{\mathbf{H}_{p,w}^{\gamma+2}(T)}
    \leq
    N\|f\|_{\mathbf{\Lambda}_{p,w}^{\gamma}(T)}
\end{equation}
for every solution $u$ of \eqref{25.11.01.12.50}, with a constant
$N = N(d,\gamma,K_0,K,\nu,p,T)$ that is independent of $\lambda$.

Fix $\lambda_0\in\mathrm{S}$, and let
$$
\mathcal{T}_{\lambda_0} : f \mapsto u
$$
denote the solution operator associated with \eqref{25.11.01.12.50} at
$\lambda=\lambda_0$. 
Then \eqref{25.12.11.15.57} implies
\begin{equation}
\label{25.12.11.16.01}
\|\mathcal{T}_{\lambda_0} f\|_{\mathbf{H}_{p,w}^{\gamma+2}(T)}\leq N\,\|f\|_{\mathbf{\Lambda}_{p,w}^{\gamma}(T)},
\quad
\forall f\in\mathbf{\Lambda}_{p,w}^{\gamma}(T),
\end{equation}
so that $\mathcal{T}_{\lambda_0}$ is bounded with the operator norm at most $N$.
Now fix $f\in\mathbf{\Lambda}_{p,w}^{\gamma}(T)$ and
$\lambda\in[0,1]$. 
We rewrite \eqref{25.11.01.12.50} as a fixed point problem
around $\lambda_0$:
$$
\partial_t u= \mathcal{L}^{\lambda_0}u+ f + (\mathcal{L}^\lambda - \mathcal{L}^{\lambda_0})u.
$$
Equivalently, $u = \mathcal{T}_{\lambda_0}
    \big(f + (\mathcal{L}^\lambda - \mathcal{L}^{\lambda_0})u\big)$.
Define the map
$$
\Phi_{\lambda_0,\lambda}(u)
    := \mathcal{T}_{\lambda_0}
       \big(f + (\mathcal{L}^\lambda - \mathcal{L}^{\lambda_0})u\big),
    \qquad
    u\in\mathbf{H}_{p,w}^{\gamma+2}(T).
$$
Then fixed points of $\Phi_{\lambda_0,\lambda}$ are precisely the solutions of
\eqref{25.11.01.12.50} for the parameter $\lambda$.
For $u_1,u_2\in\mathbf{H}_{p,w}^{\gamma+2}(T)$ we have, by the linearity of
$\mathcal{T}_{\lambda_0}$ and \eqref{25.12.11.16.01},
\begin{equation}
\label{25.12.11.16.04}
\|\Phi_{\lambda_0,\lambda}(u_1) - \Phi_{\lambda_0,\lambda}(u_2)\|_{\mathbf{H}_{p,w}^{\gamma+2}(T)}\leq N\|(\mathcal{L}^\lambda - \mathcal{L}^{\lambda_0})(u_1-u_2)\|_{\mathbf{\Lambda}_{p,w}^{\gamma}(T)}.
\end{equation}
Since the coefficients of $\mathcal{L}^\lambda$ depend linearly on $\lambda$ and are uniformly bounded in $\lambda\in[0,1]$, there exists a constant
$C_0=C_0(d,\gamma,K,K_0,\nu,p,T)>0$ such that
\begin{equation}
\label{25.12.11.16.05}
\|(\mathcal{L}^\lambda - \mathcal{L}^{\lambda_0})v\|
    _{\mathbf{\Lambda}_{p,w}^{\gamma}(T)}
\leq
C_0\,|\lambda-\lambda_0|\,
\|v\|_{\mathbf{H}_{p,w}^{\gamma+2}(T)},
\quad
\forall v\in\mathbf{H}_{p,w}^{\gamma+2}(T),
\ \lambda,\lambda_0\in[0,1].
\end{equation}
Combining \eqref{25.12.11.16.04} and \eqref{25.12.11.16.05} gives
$$
\|\Phi_{\lambda_0,\lambda}(u_1) - \Phi_{\lambda_0,\lambda}(u_2)\|_{\mathbf{H}_{p,w}^{\gamma+2}(T)}
\leq N C_0|\lambda-\lambda_0|\|u_1-u_2\|_{\mathbf{H}_{p,w}^{\gamma+2}(T)}.
$$
Choose $\varepsilon_0>0$ such that $N C_0\,\varepsilon_0 < 1$,
where $\varepsilon_0=\varepsilon_0(d,\gamma,K_0,K,\nu,p,T)$ is independent of $\lambda_0$.  
Then, for all $\lambda\in[0,1]$ with $|\lambda-\lambda_0|<\varepsilon_0$, the map $\Phi_{\lambda_0,\lambda}$ is a contraction on $\mathbf{H}_{p,w}^{\gamma+2}(T)$. 
By the Banach fixed-point theorem, $\Phi_{\lambda_0,\lambda}$ admits a unique fixed point $u\in\mathbf{H}_{p,w}^{\gamma+2}(T)$, which satisfies \eqref{25.11.01.12.50}.
Hence,
$$
(\lambda_0-\varepsilon_0,\lambda_0+\varepsilon_0)\cap[0,1]\subset \mathrm{S}.
$$
Thus every point of $\mathrm{S}$ has a neighborhood of radius $\varepsilon_0$
contained in $\mathrm{S}$, with $\varepsilon_0$ independent of the point.

\medskip

\noindent
\textbf{Step 3. Covering argument.}
Since $0\in\mathrm{S}$ and $\varepsilon_0>0$ is fixed, we have
$[0,\varepsilon_0]\cap[0,1]\subset\mathrm{S}$.  
Iterating this argument, we obtain
$$
[0,k\varepsilon_0]\cap[0,1]\subset\mathrm{S}
\quad\text{for each }k\in\mathbb{N}.
$$
Choosing $k_0\in\mathbb{N}$ with $k_0\varepsilon_0\ge1$ yields
$[0,1]\subset\mathrm{S}$, and hence $\mathrm{S}=[0,1]$.

This proves the existence of a unique solution
$u\in\mathbf{H}_{p,w}^{\gamma+2}(T)$ to \eqref{25.11.01.12.50} for $\lambda=1$.
Together with the a priori estimate and uniqueness, this completes the proof. 
\end{proof}

\subsection{Proof of Theorem \ref{25.10.20.19.28} with non-zero initial data}

We now turn to the case of non-zero initial data.  
The following trace theorem plays a crucial role in our analysis.  
It shows that $\mathbf{H}_{p,w}^{\gamma+2}(T)$ admits a well-defined and  continuous trace at $t=0$, and moreover that every element of the space $\Lambda_{p}^{\gamma+2,w}(\mathbb{R}^d)$ arises as such a trace.  
In particular, this identifies $\Lambda_{p}^{\gamma+2,w}(\mathbb{R}^d)$ as the \emph{optimal} initial data space for the regularity class $\mathbf{H}_{p,w}^{\gamma+2}(T)$, and provides the exact mechanism by which the non-zero initial value problem can be reduced to the zero initial data case.

\begin{thm}[Trace theorem]
\label{25.11.01.13.31}
    Let $\eta\in(0,\infty)$, $p\in(1,\infty]$, $T>0$, and $w\in A_p$ with $[w]_{A_p}\leq K_0$.
    \begin{enumerate}[(i)]
        \item The space $\mathbf{H}_{p,w}^{\eta+2}(T)$ is continuously embedded into $C([0,T];\Lambda_{p}^{\eta+2,w}(\mathbb{R}^d))$; that is,
        $$
        \sup_{t\in[0,T]}\|u(t,\cdot)\|_{\Lambda_{p}^{\eta+2,w}(\mathbb{R}^d)}\leq N\|u\|_{\mathbf{H}_{p,w}^{\eta+2}(T)}.
        $$
        \item For every $u_0 \in \Lambda_{p}^{\eta+2,w}(\mathbb{R}^d)$, there exists $v \in \mathbf{H}_{p,w}^{\eta+2}(T)$ satisfying $v(0,\cdot) = u_0$, and
        \begin{equation}
        \label{25.11.01.13.33}
        \|v\|_{\mathbf{H}_{p,w}^{\eta+2}(T)}\leq N\|u_0\|_{\Lambda_{p}^{\eta+2,w}(\mathbb{R}^d)}.
        \end{equation}
    \end{enumerate}
    Here $N$ depends on $K_0,p$, and $T$.
\end{thm}
\begin{proof}
Since the results for the case $p=\infty$ follow directly from the definition of $\mathbf{H}_{p,w}^{\eta+2}(T)$, we focus on the case $p\in(1,\infty)$.
By \cite[Lemma 3.1]{C24} with $\phi(\lambda) = \lambda^{\eta}$ and $\gamma = 2$, we have
\begin{equation}
\label{25.12.30.11.15}
(\Lambda^{\eta+2}(\mathbb{R}^d), \Lambda^{\eta}(\mathbb{R}^d))_{W^{1/p},p}
= \Lambda_{p}^{\eta+2,w}(\mathbb{R}^d),
\end{equation}
where $W(t) := \int_0^t w(s)\,\mathrm{d}s$, and $(X,Y)_{W^{1/p},p}$ denotes the generalized real interpolation space (see \cite{CLSW25}).
With the characterization of the generalized real interpolation space \eqref{25.12.30.11.15}, $(i)$ and $(ii)$ follow directly from \cite[Theorem 1.8]{C24} and \cite[Theorem 1.5]{CLSW25}, respectively.
This completes the proof of the theorem.
\end{proof}

\medskip

\begin{proof}[Proof of Theorem \ref{25.10.20.19.28}]
    The uniqueness follows directly from Theorem \ref{25.10.25.16.21}.

    For the existence, we first obtain $v \in \mathbf{H}_{p,w}^{\gamma+2}(T)$ satisfying $v(0,\cdot) = u_0$ and \eqref{25.11.01.13.33} by applying Theorem \ref{25.11.01.13.31}-($ii$).  
Then $\mathcal{L}v - \partial_t v \in \mathbf{\Lambda}_{p,w}^{\gamma}(T)$, where
$$
\mathcal{L} u := \sum_{i,j=1}^{d}a_{ij} D_{ij} u + \sum_{i=1}^{d}b_i D_i u + c u.
$$
Hence, by Theorem \ref{25.10.25.16.21}, there exists a unique solution $\bar{u} \in \mathbf{H}_{p,w}^{\gamma+2}(T)$ to
\begin{equation*}
\partial_t \bar{u} = \mathcal{L}\bar{u} + (f+\mathcal{L}v- \partial_t v),
\quad \text{with the zero initial condition.}
\end{equation*}
Let $u := \bar{u} + v$.  
Then $u$ is a solution to
$$
\partial_t u = \mathcal{L} u + f,
\qquad
u(0,\cdot) = u_0.
$$
The desired estimates follow directly from Theorems \ref{25.10.25.16.21} and \ref{25.11.01.13.31}.  
This completes the proof of the theorem.
\end{proof}





\end{document}